\documentclass[12pt,leqno]{article}
\usepackage{amsfonts,amsthm,amsmath}

\usepackage{amssymb}
\usepackage{color} 

\theoremstyle{plain}
\newtheorem{thm}{Theorem}[section]
\newtheorem{prop}{Proposition}[section]

\newtheorem{cor}{Corollary}[section]

\theoremstyle{definition}
\newtheorem{df}{Definition}[section]

\newtheorem{ex}{Example}[section]

\newcommand{\RR}{\mathbb{R}}

\begin{document}

\title{A construction of spherical $3$-designs
}
\author{
Tsuyoshi Miezaki
\thanks{Faculty of Education, University of the Ryukyus, Okinawa  
903-0213, Japan 
miezaki@edu.u-ryukyu.ac.jp
(Corresponding author)
}
}

\date{}
\maketitle

\begin{abstract}
We give a construction for spherical $3$-designs. 
This construction is a generalization of 
Bondarenko's work.
\end{abstract}

{\small
\vspace{-0.15in}

\quad

\noindent
{\it Classification codes according to UDC}. 
519.1
\\ \quad
}

\section{Introduction}
This paper is inspired by \cite{B}, which 
gives an optimal antipodal spherical $(35,240,1/7)$ code
whose vectors form a spherical $3$-design. 
To explain our results, 
we review the concept of spherical $t$-designs and \cite{B}. 

First, we explain the concept of spherical $t$-designs.
\begin{df}[\cite{DGS}]
For a positive integer $t$, a finite nonempty set $X$ in the unit sphere
\[
S^{d} = \{{ x} = (x_1, \ldots , x_{d+1}) \in \RR ^{d+1}\mid 
x_1^{2}+ \cdots + x_{d+1}^{2} = 1\}
\]
is called a spherical $t$-design in $S^{d}$ if the following condition is satisfied:
\[
\frac{1}{|X|}\sum_{{ x}\in X}f({ x})=\frac{1}{|S^{d}|}
\int_{S^{d}}f({ x})d\sigma ({ x}), 
\]
for all polynomials $f({ x}) = f(x_1, \ldots ,x_{d+1})$ 
of degree not exceeding $t$. Here, the righthand side involves the surface integral over the sphere 
and $|S^{d}|$, the volume of sphere $S^{d}$. 

\end{df}
The meaning of spherical $t$-designs is that the average value of the integral of any polynomial of degree up to $t$ on the sphere can be replaced by its average value over a finite set on the sphere. 

The following is an equivalent condition of the 
antipodal spherical designs: 
\begin{prop}[\cite{V}]\label{prop:V}
An antipodal set $X = \{x_1,\ldots, x_N\}$ in $S^d$ 
forms a spherical $3$-design if and only if
\[
\frac{1}{|X|^2}\sum_{x_i,x_j\in X}(x_i,x_j)^2=\frac{1}{d+1}. 
\]

An antipodal set $X = \{x_1,\ldots, x_N\}$ in $S^d$ 
forms a spherical $5$-design if and only if
\begin{align}\label{eqn:4}
\left\{
\begin{array}{l}
\displaystyle\frac{1}{|X|^2}\sum_{x_i,x_j\in X}(x_i,x_j)^2=\frac{1}{d+1},\\
\displaystyle\frac{1}{|X|^2}\sum_{x_i,x_j\in X}(x_i,x_j)^4=\frac{3}{(d+3)(d+1)}. 
\end{array}
\right.
\end{align}

\end{prop}

Next, we review \cite{B}. 
Let
\[
\Delta=\sum_{j=1}^{d+1}\frac{\partial^2}{\partial x_j^2}. 
\]
We say that a polynomial $P$ in $\RR^{d+1}$ is harmonic if 
$\Delta P=0$. 
For integer $k\geq 1$, 
the restriction of a homogeneous harmonic polynomial of degree 
$k$ to $S^{d}$ is called a spherical harmonic of degree $k$. 
We denote by $\mbox{Harm}_k(S^d)$ the vector space of 
the spherical harmonics of degree $k$. 
Note that (see for example \cite{V})
\[
\dim\mbox{Harm}_k(S^d)=\frac{2k+d-1}{k+d-1}\binom{d+k-1}{k}.
\]

For $P,Q\in \mbox{Harm}_k(S^d)$, 
we denote by $\langle P,Q\rangle$ the usual inner product
\[
\langle P,Q\rangle:=\int_{S^d}P(x)Q(x)d\sigma(x), 
\]
where $d\sigma(x)$ is a normalized Lebesgue measure on the 
unit sphere $S^d$. 
For $x\in S^d$, 
there exists $P_x\in \mbox{Harm}_k(S^d)$ such that 
\[
\langle P_x,Q\rangle=Q(x)\ \mbox{for all }Q\in \mbox{Harm}_k(S^d). 
\]
It is known that 
\[
P_x(y)=g_{k,d}((x,y)), 
\]
where $g_{k,d}$ is a Gegenbauer polynomial. 
Let 
\[
G_x=\frac{P_x}{g_{k,d}(1)^{1/2}}. 
\]
We remark that 
\[
\langle G_x,G_y\rangle=\frac{g_{k,d}((x,y))}{g_{k,d}(1)}. 
\]
(For a detailed explanation of 
Gegenbauer polynomials, see \cite{V}.) 
Therefore, 
if we have a set 
$X=\{x_1,\ldots,x_N\}$ in $S^d$, 
then we obtain the set 
$G_X=\{G_{x_1},\ldots,G_{x_N}\}$ in $S^{\dim{\rm Harm}_k(S^{d})-1}$. 

Let $X=\{x_1,\ldots,x_{120}\}$ be an arbitrary subset of $240$
normalized minimum vectors of the $E_8$ lattice such that 
no pair of antipodal vectors is present in $X$. 
Set $P_x(y)=g_{2,7}((x,y))$. 
\cite{B} showed that $G_X \cup -G_{X}$ is an optimal antipodal spherical $(35,240,1/7)$ code whose vectors form a spherical $3$-design, 
where
\[
-G_X:=\{-G_{x}\mid G_x\in G_X\}. 
\]
Furthermore, \cite{B} showed that
$G_X \cup -G_{X}$ is a spherical $3$-design, 
using the special properties of the $E_8$ lattice. 
However, this fact is an example that extends to a more general setting as follows. 
The spherical $3$-design obtained by Bondarenko in \cite{B} 
is a special case of 
our main result, which is presented as the following theorem. 
\begin{thm}\label{thm:main}
Let $X$ be a finite subset of sphere $S^d$ satisfying the 
condition {\rm (\ref{eqn:4})}. 
We set $P_x(y)=g_{2,d}((x,y))$. 
Then $G_X\cup -G_X$ is a spherical $3$-design 
in $S^{\dim{\rm Harm}_2({S^d})-1}$. 
\end{thm}
We denote by $\widetilde{G_X}$ the set $G_X\cup -G_X$
defined in Theorem \ref{thm:main}. 
\begin{cor}\label{cor:main}
\begin{enumerate}
\item 
Let $X$ be a spherical $4$-design in $S^d$. 
Then $\widetilde{G_X}$ is a spherical $3$-design 
in $S^{\dim{\rm Harm}_2({S^d})-1}$. 

\item 
Let $X$ be a spherical $4$-design in $S^d$ 
and an antipodal set. 
Let $X'$ be an arbitrary subset of $X$ with $|X'|=|X|/2$ 
such that no pair of antipodal vectors is present in $X'$. 
Then $\widetilde{G_{X'}}$ is a spherical $3$-design 
in $S^{\dim{\rm Harm}_2({S^d})-1}$. 
\end{enumerate}
\end{cor}

In section \ref{sec:proof}, we give a 
proof of Theorem \ref{thm:main}. 
In section \ref{sec:ex}, we give some 
examples.

\section{Proof of Theorem \ref{thm:main}}\label{sec:proof}

Now, we give the proof of Theorem \ref{thm:main}. 
\begin{proof}[Proof of Theorem \ref{thm:main}]

Let $X=\{x_1,\ldots,x_N\}$ be in $S^d$ and 
$G_X=\{G_{x_1},\ldots,G_{x_N}\}$ be in $\mbox{Harm}_2(S^d)$. 
By Proposition \ref{prop:V}, we have 
\begin{align*}
\left\{
\begin{array}{l}
\displaystyle\frac{1}{|X|^2}\sum_{x_i,x_j\in X}(x_i,x_j)^2=\frac{1}{d+1},\\
\displaystyle\frac{1}{|X|^2}\sum_{x_i,x_j\in X}(x_i,x_j)^4=\frac{3}{(d+3)(d+1)}, 
\end{array}
\right.
\end{align*}
since $X$ is a spherical $4$-design. 
We have the following Gegenbauer polynomial of degree $2$ on $S^d$: 
\begin{align*}
g_{2,d}(x)=\frac{d+1}{d}x^2-\frac{1}{d}. 
\end{align*}

It is enough to show that 
\[
\frac{1}{|X|^2}\sum_{x_i,x_j\in X}\langle G_{x_i},G_{x_j}\rangle^2=\frac{2}{d(d+3)}
\]
since
\[
\dim\mbox{Harm}_2(S^d)=\frac{d+3}{d+1}\binom{d+1}{2}=\frac{d(d+3)}{2} 
\]
and $G_X\cup -G_X$ is an antipodal set. 
We remark that if $X$ is a spherical $t$-design, then 
$X\cup -X$ is also a spherical $t$-design. 

In fact, 
\begin{align*}
\frac{1}{|X|^2}\sum_{x_i,x_j\in X}\langle G_{x_i},G_{x_j}\rangle^2
&=\frac{1}{|X|^2}\sum_{x_i,x_j\in X}g_{2,d}((x_i,x_j))^2\\
&=\frac{1}{|X|^2}\sum_{x_i,x_j\in X}\left(\frac{(d+1)^2}{d^2}(x_i,x_j)^4
-2\frac{(d+1)}{d^2}(x_i,x_j)^2+\frac{1}{d^2}\right)\\
&=\frac{(d+1)^2}{d^2}\frac{3}{(d+3)(d+1)}
-2\frac{(d+1)}{d^2}\frac{1}{d+1}+\frac{1}{d^2}\\
&=\frac{2}{d(d+3)}. 
\end{align*}
Therefore, if $X=\{x_1,\ldots,x_N\}$ is a spherical $4$-design, then 
$G_X \cup -G_X$ is a spherical $3$-design. 

\end{proof}

\section{Examples}\label{sec:ex}

In this section, 
we give some examples of using Theorem \ref{thm:main}. 

First we recall the concept of 
a strongly perfect and spherical 
$( d +1,N ,a )$ code. 
\begin{df}[\cite{V}]
A lattice $L$ is called strongly perfect if 
the minimum vectors of $L$ form a spherical $5$-design. 
\end{df}

\begin{df}[\cite{CS}]
An antipodal set 
$X = \{ x_1,\ldots,x_N\}$ in $S^d$ 
is called an antipodal spherical 
$( d +1,N ,a )$ code 
if $|(x_i,x_j )|\leq a$ for some $a > 0$ 
and all $x_i,x_i\in X$, $i\neq j$, are not antipodal. 

\end{df}

Next we give some examples. 

\begin{ex}
The strongly perfect lattices whose ranks are
less than $12$ have been classified \cite{{NV1},{NV2}}. 
Such lattices whose ranks are greater than $1$ are as follows: 
\[
A_2,\ D_4,\ E_6,\ E_6^{\sharp},\ E_7,\ E_7^{\sharp},
\ E_8,\ K_{10},\ K_{10}^{\sharp},\ CT_{12}. 
\]
(For a detailed explanation, see \cite{{NV1},{NV2}}.)
Let $L$ be one of the above lattices 
and $X$ be the minimum vectors of $L$. 
Then, let $X'$ be an arbitrary subset of $X$ with $|X'|=|X|/2$ 
such that no pair of antipodal vectors is present in $X'$. 

By Corollary \ref{cor:main}, 
$G_{X}\cup -G_{X}$ is a spherical $3$-design in 
$S^d$, where $d$ is as follows: 
\[
\begin{array}{c||c|c|c}
L&(d+1,N,a)\ {\rm code}&|(x_i,x_j)|&|\langle G_{x_i},G_{x_j}\rangle| \\\hline\hline
A_2&(2,6,1/2)&\{1/2\}&\{1/2\}\\\hline
D_4&(9,24,1/3)&\{0,1/2\}&\{0,1/3\}\\\hline
E_6&(20,72,1/5)&\{0,1/2\}&\{1/10,1/5\}\\\hline
E_6^\sharp&(20,54,1/8)&\{1/4,1/2\}&\{1/10,1/8\}\\\hline
E_7&(27,126,1/6)&\{0,1/2\}&\{1/8,1/6\}\\\hline
E_7^\sharp&(27,56,1/27)&\{1/3\}&\{1/27\}\\\hline
E_8&(35,240,1/7)\ \cite{B}&\{0,1/2\}&\{1/7\}\\\hline
K_{10}&(54,276,1/6)&\{0,1/4,1/2\}&\{1/24,1/9,1/6\}\\\hline
K_{10}^\sharp&(54,54,1/6)&\{1/8,1/4,1/2\}&\{1/24,3/32,1/6\}\\\hline
CT_{12}&(77,756,2/11)&\{0,1/4,1/2\}&\{1/44,1/11,2/11\}\\
\end{array}
\]

\end{ex}

\begin{ex}
Let $X$ be the minimum vectors of the Barnes--Wall lattice of rank $16$, 
and 
let $X'$ be an arbitrary subset of $X$ with $|X'|=|X|/2$ 
such that no pair of antipodal vectors is present in $X'$. 
We remark that $X$ is a spherical $7$-design. 

By Corollary \ref{cor:main}, 
$G_{X}\cup -G_{X}$ is a spherical $3$-design in 
$S^d$, where $d$ is as follows: 
\[
\begin{array}{c||c|c|c}
L&(d+1,N,a)\ {\rm code}&|(x_i,x_j)|&|\langle G_{x_i},G_{x_j}\rangle| \\\hline\hline
{\rm BW16} &(135,4320,1/5)&\{0,1/4,1/2\}&\{0,1/15,1/5\}\\
\end{array}
\]

\end{ex}

\begin{ex}
Let $X$ be the minimum vectors of the Leech lattice, and 
let $X'$ be an arbitrary subset of $X$ with $|X'|=|X|/2$ 
such that no pair of antipodal vectors is present in $X'$. 
We remark that $X$ is a spherical $11$-design. 

By Corollary \ref{cor:main}, 
$G_{X}\cup -G_{X}$ is a spherical $3$-design in 
$S^d$, where $d$ is as follows: 
\[
\begin{array}{c||c|c|c}
L&(d+1,N,a)\ {\rm code}&|(x_i,x_j)|&|\langle G_{x_i},G_{x_j}\rangle| \\\hline\hline
{\rm Leech} &(299,196560,5/23)&\{0,1/4,1/2\}&\{1/46,1/23,5/23\}\\
\end{array}
\]

\end{ex}




\section*{Acknowledgments}

This work was supported by JSPS KAKENHI (18K03217).



\end{document}